\documentclass[]{article}

\usepackage{comment}
\usepackage{amsmath,amsthm}
\usepackage{hyperref}
\usepackage{url}
\usepackage[english]{babel}
\usepackage{fancyhdr}
\usepackage{pgfplots}
\usepackage{cleveref}
\usepackage{listings}
\usepackage{subcaption}

\pgfkeys{/pgf/declare function={arctanh(\x) = (1/2)*ln((1+\x)/(1-\x));}}
\usetikzlibrary{arrows}
\usepackage{float}
\newfloat{lstfloat}{htbp}{lop}
\floatname{lstfloat}{Listing}
\newcommand{\myQ}{Q(r,s,t)}
\DeclareMathOperator{\CDF}{CDF}
\DeclareMathOperator{\PDF}{PDF}

\newtheorem{theorem}{Theorem}
\newtheorem{lemma}[theorem]{Lemma}
\definecolor{xdxdff}{rgb}{0.6588235294117647,0.6588235294117647,0.6588235294117647}
\definecolor{zzttqq}{rgb}{0.26666666666666666,0.26666666666666666,0.26666666666666666}
\definecolor{qqqqff}{rgb}{0.3333333333333333,0.3333333333333333,0.3333333333333333}
\definecolor{cqcqcq}{rgb}{0.7529411764705882,0.7529411764705882,0.7529411764705882}
\title{Triangle Inscribed-Triangle Picking}
\author{Arman Maesumi}
\begin{document}
\maketitle

\paragraph{Abstract.}
Given a triangle $ABC$, we derive the probability distribution function and the moments of the area of an inscribed triangle $RST$ whose vertices are uniformly distributed on $AB, BC$ and $CA$. The theoretical results are confirmed by a Monte Carlo simulation.


\paragraph{AMS Subject Classifiction.}
60D05.

\section{Introduction}

In 1865, James Joseph Sylvester proved \cite{MathematicalQuestions} that the average area of a random triangle, whose vertices are picked inside a given triangle of unit area, is equal to ${1}/{12}$. This problem, originally proposed by S. Watson, and  known as Triangle Triangle Picking, is one of the earliest examples of Geometric Probability \cite{Solomon}. Many similar problems  have been proposed \cite{Probability, CNSS}, including Sylvester's own four-point problem \cite{SylvesterFour} which asks for the probability that four random points in a convex shape have a convex hull which is a quadrilateral. Problems involving properties of inscribed geometric figures have also been studied; for example,  questions related to the average distance of inscribed points appear in \cite{BBKW}, while in \cite{RandomTriangles} the average area and perimeter of a triangle inscribed in a circle is found. Here we consider a class of  such problems where the interior polygon has its vertices on the edges of the base convex polygon, with one vertex per side. In particular we look at the properties of a random triangle that is inscribed in a fixed triangle.

\section{An Application of Barycentric Coordinates}\label{sec:one}

A simple and effective way of describing triangles within triangles is to use the barycentric coordinates.
 Suppose the vertices of a triangle are denoted by the vectors $\vec{A}, \vec{B}, \vec{C}$.
The barycentric coordinates \cite{Barycentric} of a point $\vec P$, with respect to the triangle $ABC$,  is $(\alpha,\beta,\gamma)$ if $\vec P = \alpha \vec A + \beta \vec B + \gamma \vec C$, and $\alpha + \beta + \gamma = 1$. Bottema's theorem \cite{Bottema} gives the area of a triangle if the barycentric coordinates of its vertices are known  with respect to another triangle. 

\begin{theorem}
	[Bottema] Let $|\Delta ABC|$ represent the signed area of triangle ABC.  Assume the vertices $P_{i}$ of a triangle $P_{1}P_{2}P_{3}$ have barycentric coordinates $(x_{i}, y_i, z_i)$,  with respect to the triangle  $ABC$, then,
	
	\begin{equation}
	|\Delta P_1P_2P_3| = \det \left[ \begin{array}{ccc}
	x_1 & y_1 & z_1 \\
	x_2 & y_2 & z_2 \\
	x_3 & y_3 & z_3 \end{array} \right]|\Delta ABC|.
	\end{equation}

\end{theorem}

By using the above theorem we can easily calculate the moments of the area of the inscribed triangle.

\begin{theorem}
	Given a triangle $ABC$, if three points $R,S,$ and $T$ are chosen uniformly on the sides $AB, BC,$ and $ CA$ respectively then the average area of $RST$ is one-fourth of the area of $ABC$. 
\end{theorem}
\begin{proof}
 Consider an inscribed triangle whose vertices $R,S,T,$ are defined as
\begin{equation}
\begin{cases}
\vec{R} = \vec{B}+r \   \overrightarrow{BC}\\
\vec{S} = \vec{C}+s \  \overrightarrow{CA}\\
\vec{T} = \vec{A}+t \  \overrightarrow{AB}
\end{cases}
\end{equation}
 where $r,s,t$ are uniformly distributed random numbers in $[0,1]$. 

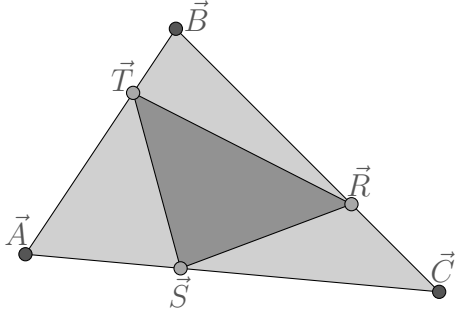
\begin{figure}[h]
	\centering
	\hspace*{4.2em}\begin{tikzpicture}[line cap=round,line join=round,>=triangle 45,x=.5cm,y=.5cm]
	\clip(1.0241471171800514,0.2292121556504571) rectangle (17.231040580033817,9.46039625697576);
	\fill[color=zzttqq,fill=zzttqq,fill opacity=0.25] (2.,2.) -- (6.,8.) -- (13.,1.) -- cycle;
	
	\fill[color=zzttqq,fill=qqqqff,fill opacity=0.5] (4.862853203572577,6.294279805358864)-- (6.124763458659939,1.6250215037581874) -- (10.668356437672603,3.331643562327397) -- cycle;
	
	\draw (2.,2.)-- (6.,8.);
	\draw (6.,8.)-- (13.,1.);
	\draw (13.,1.)-- (2.,2.);
	\draw (4.862853203572577,6.294279805358864)-- (6.124763458659939,1.6250215037581874);
	\draw (4.862853203572577,6.294279805358864)-- (10.668356437672603,3.331643562327397);
	\draw (6.124763458659939,1.6250215037581874)-- (10.668356437672603,3.331643562327397);
	\begin{scriptsize}
	\draw [fill=qqqqff] (2.,2.) circle (2.5pt);
	\draw[color=qqqqff] (1.7522750261982737,2.66109408467515) node {\large $\vec{A}$};
	\draw [fill=qqqqff] (6.,8.) circle (2.5pt);
	\draw[color=qqqqff] (6.545338410625247,8.343120192771338) node {\large $\vec{B}$};
	\draw [fill=qqqqff] (13.,1.) circle (2.5pt);
	\draw[color=qqqqff] (13.098172465166735,1.6733029783652823) node {\large $\vec{C}$};
	%\draw[color=zzttqq] (3.833890919088849,5.155026344078715) node {\large $c$};
	%\draw[color=zzttqq] (9.677230177692866,4.729720005636019) node {\large $a$};
	%\draw[color=zzttqq] (7.282570576461156,1.4474645676543394) node {\large $b$};
	\draw [fill=xdxdff] (4.862853203572577,6.294279805358864) circle (2.5pt);
	\draw[color=qqqqff] (4.5769614956357,6.751140628966698) node {\large $\vec{T}$};
	\draw [fill=xdxdff] (6.124763458659939,1.6250215037581874) circle (2.5pt);
	\draw[color=qqqqff] (6.080617880862229,1.0017271977559027) node {\large $\vec{S}$};
	\draw [fill=xdxdff] (10.668356437672603,3.331643562327397) circle (2.5pt);
	\draw[color=qqqqff] (10.857479183664191,3.954709996646954) node {\large $\vec{R}$};
	%\draw[color=black] (5.359446263502872,3.768157849156879) node {\large $f$};
	%\draw[color=black] (7.772597444666872,4.9516189648235125) node {\large $g$};
	%\draw[color=black] (8.512260641958521,2.34430619437046) node {\large $h$};
	\end{scriptsize}
	\end{tikzpicture}
\caption{Triangle $ABC$, and an inscribed triangle $RST$.}
\end{figure}
In this case, the points $R,S,T$ are respectively given by barycentric coordinates $(0, r, 1-r),(1-s, 0, s)$ and $(t, 1-t, 0)$. Now we define $\myQ$ as the quotient $|\Delta RST|/|\Delta ABC|$. Therefore, by Bottema's theorem
\begin{equation}
 \myQ = \det \left[ \begin{array}{ccc}
 0 & r & 1-r \\
 1-s & 0 & s \\
 t & 1-t & 0 \end{array} \right] = rst + (1-r)(1-s)(1-t).
\end{equation} 
Now we will set out to calculate $E[Q]$, the expected value of $\myQ$. The expected value of $rst$, and $(1-r)(1-s)(1-t)$, can be represented by the product of the expected values of $r, s, t$. Specifically,
 \begin{equation}
 \begin{split}
 E[\myQ] &= E[rst + (1-r)(1-s)(1-t)] \\ 
 &= \int_{0}^{1} \int_{0}^{1} \int_{0}^{1} (rst + (1-r)(1-s)(1-t))\, dr\,ds\,dt \\
 &= \left(\dfrac{1}{2}\right)^3 + \left(\dfrac{1}{2}\right)^3 = \dfrac{1}{4}.
 \end{split}
 \end{equation}
As a result, $E[|\Delta RST|] = \dfrac{1}{4}|\Delta ABC|$.
\end{proof}

\subsection{The \boldmath{$n^{th}$}  moment, \boldmath{$E[Q^n(r,s,t)]$}}
To derive the $n^{th}$ moment of the area, we expand $Q^n(r,s,t)$ using the Binomial Theorem,
\begin{equation}
	\begin{split}
		Q^n(r,s,t) &= [rst + (1-r)(1-s)(1-t)]^n \\
		& = \sum_{k=0}^{n} \dfrac{n!}{k!(n-k)!}(rst)^{n-k}((1-r)(1-s)(1-t))^k \\
		& = \sum_{k=0}^{n} \dfrac{n!}{k!(n-k)!}r^{n-k}(1-r)^ks^{n-k}(1-s)^kt^{n-k}(1-t)^k.
	\end{split}
\end{equation}
The average value of $(rst)^{n-k}((1-r)(1-s)(1-t))^k$, can be found using the Euler  beta function \cite{Beta}
% $r^{n-k}(1-r)^ks^{n-k}(1-s)^kt^{n-k}(1-t)^k$, can be found using the Beta Function [8], therefore,
\begin{equation}
	%\begin{split}
	E[ r^{n-k}(1-r)^k ] = 
	\int_{0}^{1}r^{n-k}(1-r)^kdr    
	%=\int_{0}^{1}r^v(1-r)^wdr \\ 
	%&= \dfrac{\Gamma(v+1)\Gamma(w+1)}{\Gamma(v+w+2)}
	%=\dfrac{v!w!}{(v+w+1)!}=
	=\dfrac{(n-k)!k!}{(n+1)!}.
	%\end{split}
\end{equation}
Thus, $\mu_n$, the $n^{th}$ moment of $\myQ$, can now be expressed as
\begin{equation}\label{nthmoment}
	\begin{split}
		\mu_n=E[Q^n(r,s,t)] &= \sum_{k=0}^{n}\dfrac{n!}{k!(n-k)!}\left(\frac{(n-k)!k!}{(n+1)!}\right)^3 \\
		&=\dfrac{1}{(n+1) (n+1)!^2}\sum_{k=0}^{n}(n-k)!^2k!^2.
	\end{split}
\end{equation} 
Therefore, $E[|\Delta RST|^n] =\mu_n |\Delta ABC|^n$.
We recorded the sum in \eqref{nthmoment} as Sloane integer sequence A279055, \cite{Arman}. In particular, the first few moments are as shown in Table $\ref{table:NineMoments}$.

\begin{comment}
\begin{table}[H]
\centering
\begin{tabular}{ c | c c c c c c c c c c}
$n$ & 1 & 2 & 3 & 4 & 5 & 6 & 7 & 8 & 9 \\\\
$E[Q^n]$ & $\dfrac{1}{4}$ & $\dfrac{1}{12}$ & $\dfrac{5}{144}$ & $\dfrac{31}{1800}$ & $\dfrac{7}{720}$ & $\dfrac{1063}{176400}$ & $\dfrac{403}{100800}$ & $\dfrac{211}{75600}$ & $\dfrac{143}{70560}$
\end{tabular}
\caption{The first nine moments of $\myQ$.}\label{table:NineMoments}
\end{table}
\end{comment}

\begin{table}[H]
	\centering
	\begin{tabular}{ c | c c c c c c c c}
		$n$ & 1 & 2 & 3 & 4 & 5 & 6 & 7 \\\\
		$\mu_n$ & $\dfrac{2}{2\cdot 2!^2}$ & $\dfrac{9}{3\cdot 3!^2}$ & $\dfrac{80}{4\cdot 4!^2}$ & $\dfrac{1240}{5\cdot 5!^2}$ & $\dfrac{30240}{6 \cdot 6!^2}$ & $\dfrac{1071504}{7 \cdot 7!^2}$ & $\dfrac{51996672}{8\cdot 8!^2}$ 
	\end{tabular}
	\caption{The first seven moments of $\myQ$.}\label{table:NineMoments}
\end{table}

\begin{comment}
10^2: 1.3236726778772518E-2
10^3:  3.009873250385698E-3
10^4 : 1.1399255511913924E-3
10^5 : 3.809102729615965E-4
10^6 : 9.636680416315157E-5  1.1621989398601817E-4
10^7 : 3.85157931698632E-5
10^8 : 1.1853210084296628E-5
\end{comment}

\section{A Monte Carlo Simulation of the Probability Density Function}
A Monte Carlo simulation  \cite{MC1,MC2}  was conducted to numerically study and validate the theoretical findings for  the distribution of the area of a randomly generated inscribed triangle. 
The output of the simulation is the experimental probability density function as depicted in \ref{fig:SharkFin}. We derive the elementary functions that produce this curve in the next section.

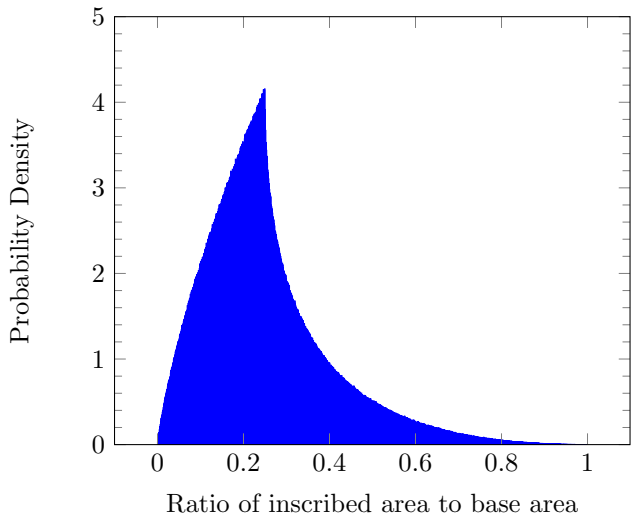
\begin{figure}[h] 
	\centering
	\begin{tikzpicture}
	\begin{axis}[
	xlabel=Ratio of inscribed area to base area,
	ylabel=Probability Density,
	ymin=0, ymax=5,
	minor y tick num = 4,
	area style,
	]
	\addplot+[ybar interval,mark=no] plot coordinates {(0.0000,0.0278) (0.0010,0.0652) (0.0020,0.1017) (0.0030,0.1407) (0.0040,0.1729) (0.0050,0.2079) (0.0060,0.2336) (0.0070,0.2616) (0.0080,0.3046) (0.0090,0.3204) (0.0100,0.3463) (0.0110,0.3781) (0.0120,0.4002) (0.0130,0.4315) (0.0140,0.4640) (0.0150,0.4845) (0.0160,0.5138) (0.0170,0.5467) (0.0180,0.5624) (0.0190,0.5908) (0.0200,0.6130) (0.0210,0.6432) (0.0220,0.6462) (0.0230,0.6736) (0.0240,0.7092) (0.0250,0.7315) (0.0260,0.7496) (0.0270,0.7610) (0.0280,0.7766) (0.0290,0.7940) (0.0300,0.8326) (0.0310,0.8658) (0.0320,0.8823) (0.0330,0.8985) (0.0340,0.9156) (0.0350,0.9435) (0.0360,0.9577) (0.0370,0.9916) (0.0380,1.0132) (0.0390,1.0189) (0.0400,1.0462) (0.0410,1.0815) (0.0420,1.0960) (0.0430,1.1087) (0.0440,1.1257) (0.0450,1.1520) (0.0460,1.1655) (0.0470,1.1976) (0.0480,1.2052) (0.0490,1.2090) (0.0500,1.2342) (0.0510,1.2691) (0.0520,1.2896) (0.0530,1.2953) (0.0540,1.3168) (0.0550,1.3401) (0.0560,1.3863) (0.0570,1.3929) (0.0580,1.3900) (0.0590,1.4095) (0.0600,1.4343) (0.0610,1.4574) (0.0620,1.4643) (0.0630,1.5045) (0.0640,1.4878) (0.0650,1.5242) (0.0660,1.5538) (0.0670,1.5439) (0.0680,1.5699) (0.0690,1.6225) (0.0700,1.6219) (0.0710,1.6498) (0.0720,1.6762) (0.0730,1.6532) (0.0740,1.6979) (0.0750,1.7101) (0.0760,1.7551) (0.0770,1.7465) (0.0780,1.7805) (0.0790,1.7971) (0.0800,1.8049) (0.0810,1.7768) (0.0820,1.8321) (0.0830,1.8372) (0.0840,1.8630) (0.0850,1.8791) (0.0860,1.9082) (0.0870,1.9323) (0.0880,1.9270) (0.0890,1.9590) (0.0900,1.9555) (0.0910,1.9854) (0.0920,1.9832) (0.0930,2.0045) (0.0940,2.0023) (0.0950,2.0330) (0.0960,2.0281) (0.0970,2.0552) (0.0980,2.1117) (0.0990,2.0856) (0.1000,2.1378) (0.1010,2.1366) (0.1020,2.1549) (0.1030,2.1556) (0.1040,2.1848) (0.1050,2.1780) (0.1060,2.1974) (0.1070,2.2265) (0.1080,2.2314) (0.1090,2.2365) (0.1100,2.2768) (0.1110,2.2870) (0.1120,2.2999) (0.1130,2.3507) (0.1140,2.3531) (0.1150,2.3265) (0.1160,2.3475) (0.1170,2.3954) (0.1180,2.4051) (0.1190,2.4341) (0.1200,2.4509) (0.1210,2.4744) (0.1220,2.4467) (0.1230,2.4707) (0.1240,2.4839) (0.1250,2.4814) (0.1260,2.5187) (0.1270,2.5669) (0.1280,2.5670) (0.1290,2.5643) (0.1300,2.5935) (0.1310,2.5746) (0.1320,2.6114) (0.1330,2.6337) (0.1340,2.6240) (0.1350,2.6643) (0.1360,2.6930) (0.1370,2.7004) (0.1380,2.6951) (0.1390,2.6989) (0.1400,2.7227) (0.1410,2.7517) (0.1420,2.7511) (0.1430,2.8005) (0.1440,2.7875) (0.1450,2.8066) (0.1460,2.7979) (0.1470,2.8109) (0.1480,2.8395) (0.1490,2.8626) (0.1500,2.8514) (0.1510,2.8738) (0.1520,2.9070) (0.1530,2.9148) (0.1540,2.9200) (0.1550,2.9459) (0.1560,2.9503) (0.1570,2.9972) (0.1580,2.9515) (0.1590,2.9959) (0.1600,3.0105) (0.1610,3.0545) (0.1620,3.0471) (0.1630,3.0616) (0.1640,3.0430) (0.1650,3.0697) (0.1660,3.0838) (0.1670,3.1099) (0.1680,3.1201) (0.1690,3.1306) (0.1700,3.1943) (0.1710,3.1837) (0.1720,3.1694) (0.1730,3.1914) (0.1740,3.2173) (0.1750,3.2071) (0.1760,3.1994) (0.1770,3.2430) (0.1780,3.2445) (0.1790,3.2695) (0.1800,3.2647) (0.1810,3.3349) (0.1820,3.3113) (0.1830,3.3177) (0.1840,3.3303) (0.1850,3.3510) (0.1860,3.3350) (0.1870,3.3828) (0.1880,3.3782) (0.1890,3.3853) (0.1900,3.3988) (0.1910,3.4185) (0.1920,3.4626) (0.1930,3.4515) (0.1940,3.4833) (0.1950,3.4787) (0.1960,3.5096) (0.1970,3.4728) (0.1980,3.5394) (0.1990,3.5205) (0.2000,3.5178) (0.2010,3.5202) (0.2020,3.5722) (0.2030,3.6189) (0.2040,3.6259) (0.2050,3.5821) (0.2060,3.6387) (0.2070,3.6402) (0.2080,3.6719) (0.2090,3.6511) (0.2100,3.6798) (0.2110,3.6897) (0.2120,3.7199) (0.2130,3.6685) (0.2140,3.6794) (0.2150,3.7447) (0.2160,3.7285) (0.2170,3.7672) (0.2180,3.7811) (0.2190,3.7558) (0.2200,3.7742) (0.2210,3.7937) (0.2220,3.8337) (0.2230,3.8388) (0.2240,3.8522) (0.2250,3.8617) (0.2260,3.8489) (0.2270,3.8732) (0.2280,3.8784) (0.2290,3.8845) (0.2300,3.9141) (0.2310,3.9410) (0.2320,3.9609) (0.2330,3.9655) (0.2340,3.9538) (0.2350,3.9872) (0.2360,3.9523) (0.2370,3.9758) (0.2380,4.0398) (0.2390,4.0242) (0.2400,4.0341) (0.2410,4.0383) (0.2420,4.0964) (0.2430,4.0709) (0.2440,4.0680) (0.2450,4.0994) (0.2460,4.1455) (0.2470,4.1447) (0.2480,4.1390) (0.2490,4.1564) (0.2500,3.9138) (0.2510,3.6711) (0.2520,3.5300) (0.2530,3.4386) (0.2540,3.3682) (0.2550,3.2841) (0.2560,3.2078) (0.2570,3.1727) (0.2580,3.1048) (0.2590,3.0508) (0.2600,2.9798) (0.2610,2.9392) (0.2620,2.9126) (0.2630,2.8949) (0.2640,2.7954) (0.2650,2.7781) (0.2660,2.7528) (0.2670,2.7045) (0.2680,2.6618) (0.2690,2.6161) (0.2700,2.6257) (0.2710,2.5805) (0.2720,2.5546) (0.2730,2.4947) (0.2740,2.4822) (0.2750,2.4562) (0.2760,2.4188) (0.2770,2.3931) (0.2780,2.3819) (0.2790,2.3668) (0.2800,2.3327) (0.2810,2.3092) (0.2820,2.2962) (0.2830,2.2602) (0.2840,2.2348) (0.2850,2.2101) (0.2860,2.1974) (0.2870,2.1667) (0.2880,2.1340) (0.2890,2.1160) (0.2900,2.1182) (0.2910,2.1008) (0.2920,2.0666) (0.2930,2.0339) (0.2940,2.0211) (0.2950,1.9998) (0.2960,1.9877) (0.2970,1.9702) (0.2980,1.9784) (0.2990,1.9365) (0.3000,1.9072) (0.3010,1.8990) (0.3020,1.8801) (0.3030,1.8893) (0.3040,1.8556) (0.3050,1.8622) (0.3060,1.8507) (0.3070,1.7998) (0.3080,1.7938) (0.3090,1.7731) (0.3100,1.7672) (0.3110,1.7558) (0.3120,1.7312) (0.3130,1.7369) (0.3140,1.7015) (0.3150,1.6894) (0.3160,1.6775) (0.3170,1.6725) (0.3180,1.6515) (0.3190,1.6787) (0.3200,1.6348) (0.3210,1.6192) (0.3220,1.6005) (0.3230,1.5789) (0.3240,1.5841) (0.3250,1.5670) (0.3260,1.5757) (0.3270,1.5380) (0.3280,1.5306) (0.3290,1.5279) (0.3300,1.5035) (0.3310,1.4792) (0.3320,1.4913) (0.3330,1.4546) (0.3340,1.4916) (0.3350,1.4703) (0.3360,1.4628) (0.3370,1.4228) (0.3380,1.4502) (0.3390,1.4044) (0.3400,1.4082) (0.3410,1.4113) (0.3420,1.3825) (0.3430,1.3717) (0.3440,1.3749) (0.3450,1.3470) (0.3460,1.3565) (0.3470,1.3409) (0.3480,1.3391) (0.3490,1.3331) (0.3500,1.3165) (0.3510,1.3117) (0.3520,1.3031) (0.3530,1.2845) (0.3540,1.2930) (0.3550,1.2686) (0.3560,1.2545) (0.3570,1.2469) (0.3580,1.2322) (0.3590,1.2464) (0.3600,1.2406) (0.3610,1.2165) (0.3620,1.2106) (0.3630,1.1746) (0.3640,1.1915) (0.3650,1.1707) (0.3660,1.1856) (0.3670,1.1717) (0.3680,1.1752) (0.3690,1.1361) (0.3700,1.1522) (0.3710,1.1222) (0.3720,1.1268) (0.3730,1.1179) (0.3740,1.1066) (0.3750,1.1177) (0.3760,1.0997) (0.3770,1.1016) (0.3780,1.0845) (0.3790,1.0832) (0.3800,1.0871) (0.3810,1.0722) (0.3820,1.0595) (0.3830,1.0514) (0.3840,1.0409) (0.3850,1.0262) (0.3860,1.0272) (0.3870,1.0393) (0.3880,1.0134) (0.3890,1.0183) (0.3900,1.0055) (0.3910,0.9913) (0.3920,1.0002) (0.3930,0.9805) (0.3940,0.9821) (0.3950,0.9683) (0.3960,0.9517) (0.3970,0.9513) (0.3980,0.9568) (0.3990,0.9438) (0.4000,0.9511) (0.4010,0.9214) (0.4020,0.9223) (0.4030,0.9278) (0.4040,0.9175) (0.4050,0.9026) (0.4060,0.9086) (0.4070,0.9041) (0.4080,0.8844) (0.4090,0.8953) (0.4100,0.8784) (0.4110,0.8668) (0.4120,0.8524) (0.4130,0.8719) (0.4140,0.8540) (0.4150,0.8533) (0.4160,0.8408) (0.4170,0.8474) (0.4180,0.8333) (0.4190,0.8286) (0.4200,0.8428) (0.4210,0.8322) (0.4220,0.8128) (0.4230,0.8061) (0.4240,0.8080) (0.4250,0.8051) (0.4260,0.7979) (0.4270,0.8023) (0.4280,0.8054) (0.4290,0.7808) (0.4300,0.7837) (0.4310,0.7769) (0.4320,0.7644) (0.4330,0.7750) (0.4340,0.7653) (0.4350,0.7585) (0.4360,0.7444) (0.4370,0.7278) (0.4380,0.7463) (0.4390,0.7316) (0.4400,0.7354) (0.4410,0.7339) (0.4420,0.7182) (0.4430,0.7271) (0.4440,0.7125) (0.4450,0.7101) (0.4460,0.7159) (0.4470,0.6943) (0.4480,0.7000) (0.4490,0.6978) (0.4500,0.6919) (0.4510,0.6698) (0.4520,0.6769) (0.4530,0.6819) (0.4540,0.6745) (0.4550,0.6822) (0.4560,0.6686) (0.4570,0.6826) (0.4580,0.6614) (0.4590,0.6594) (0.4600,0.6482) (0.4610,0.6458) (0.4620,0.6363) (0.4630,0.6358) (0.4640,0.6234) (0.4650,0.6408) (0.4660,0.6314) (0.4670,0.6219) (0.4680,0.6096) (0.4690,0.6110) (0.4700,0.6137) (0.4710,0.6063) (0.4720,0.5969) (0.4730,0.5981) (0.4740,0.6011) (0.4750,0.5836) (0.4760,0.5765) (0.4770,0.5629) (0.4780,0.5625) (0.4790,0.5753) (0.4800,0.5639) (0.4810,0.5744) (0.4820,0.5681) (0.4830,0.5706) (0.4840,0.5495) (0.4850,0.5569) (0.4860,0.5499) (0.4870,0.5457) (0.4880,0.5494) (0.4890,0.5359) (0.4900,0.5425) (0.4910,0.5369) (0.4920,0.5244) (0.4930,0.5156) (0.4940,0.5268) (0.4950,0.5379) (0.4960,0.5197) (0.4970,0.5152) (0.4980,0.5062) (0.4990,0.4947) (0.5000,0.5099) (0.5010,0.4982) (0.5020,0.4910) (0.5030,0.5104) (0.5040,0.4951) (0.5050,0.4996) (0.5060,0.4930) (0.5070,0.4841) (0.5080,0.5012) (0.5090,0.4863) (0.5100,0.4746) (0.5110,0.4765) (0.5120,0.4718) (0.5130,0.4645) (0.5140,0.4708) (0.5150,0.4642) (0.5160,0.4530) (0.5170,0.4628) (0.5180,0.4633) (0.5190,0.4513) (0.5200,0.4544) (0.5210,0.4484) (0.5220,0.4468) (0.5230,0.4450) (0.5240,0.4382) (0.5250,0.4310) (0.5260,0.4339) (0.5270,0.4333) (0.5280,0.4294) (0.5290,0.4290) (0.5300,0.4307) (0.5310,0.4252) (0.5320,0.4186) (0.5330,0.4199) (0.5340,0.4087) (0.5350,0.3980) (0.5360,0.4080) (0.5370,0.4037) (0.5380,0.3907) (0.5390,0.3912) (0.5400,0.4033) (0.5410,0.3867) (0.5420,0.3776) (0.5430,0.3910) (0.5440,0.3927) (0.5450,0.3873) (0.5460,0.3809) (0.5470,0.3828) (0.5480,0.3678) (0.5490,0.3708) (0.5500,0.3685) (0.5510,0.3657) (0.5520,0.3641) (0.5530,0.3655) (0.5540,0.3667) (0.5550,0.3552) (0.5560,0.3549) (0.5570,0.3552) (0.5580,0.3509) (0.5590,0.3570) (0.5600,0.3510) (0.5610,0.3419) (0.5620,0.3442) (0.5630,0.3407) (0.5640,0.3335) (0.5650,0.3311) (0.5660,0.3300) (0.5670,0.3348) (0.5680,0.3340) (0.5690,0.3233) (0.5700,0.3229) (0.5710,0.3271) (0.5720,0.3239) (0.5730,0.3124) (0.5740,0.3337) (0.5750,0.3244) (0.5760,0.3057) (0.5770,0.3052) (0.5780,0.3118) (0.5790,0.3091) (0.5800,0.3111) (0.5810,0.3093) (0.5820,0.3020) (0.5830,0.2951) (0.5840,0.2952) (0.5850,0.2943) (0.5860,0.2984) (0.5870,0.2955) (0.5880,0.2803) (0.5890,0.2844) (0.5900,0.2827) (0.5910,0.2838) (0.5920,0.2798) (0.5930,0.2793) (0.5940,0.2870) (0.5950,0.2796) (0.5960,0.2856) (0.5970,0.2737) (0.5980,0.2735) (0.5990,0.2668) (0.6000,0.2621) (0.6010,0.2624) (0.6020,0.2568) (0.6030,0.2577) (0.6040,0.2596) (0.6050,0.2674) (0.6060,0.2623) (0.6070,0.2599) (0.6080,0.2541) (0.6090,0.2458) (0.6100,0.2474) (0.6110,0.2489) (0.6120,0.2448) (0.6130,0.2493) (0.6140,0.2423) (0.6150,0.2413) (0.6160,0.2398) (0.6170,0.2436) (0.6180,0.2387) (0.6190,0.2283) (0.6200,0.2411) (0.6210,0.2373) (0.6220,0.2323) (0.6230,0.2318) (0.6240,0.2181) (0.6250,0.2230) (0.6260,0.2223) (0.6270,0.2232) (0.6280,0.2237) (0.6290,0.2280) (0.6300,0.2183) (0.6310,0.2183) (0.6320,0.2159) (0.6330,0.2147) (0.6340,0.2141) (0.6350,0.2117) (0.6360,0.2177) (0.6370,0.2097) (0.6380,0.2152) (0.6390,0.2087) (0.6400,0.2082) (0.6410,0.1972) (0.6420,0.2021) (0.6430,0.2005) (0.6440,0.1971) (0.6450,0.1949) (0.6460,0.1929) (0.6470,0.1984) (0.6480,0.1946) (0.6490,0.1923) (0.6500,0.1909) (0.6510,0.1886) (0.6520,0.1915) (0.6530,0.1925) (0.6540,0.1920) (0.6550,0.1753) (0.6560,0.1905) (0.6570,0.1829) (0.6580,0.1823) (0.6590,0.1738) (0.6600,0.1805) (0.6610,0.1725) (0.6620,0.1732) (0.6630,0.1682) (0.6640,0.1679) (0.6650,0.1694) (0.6660,0.1699) (0.6670,0.1727) (0.6680,0.1625) (0.6690,0.1653) (0.6700,0.1647) (0.6710,0.1639) (0.6720,0.1553) (0.6730,0.1551) (0.6740,0.1571) (0.6750,0.1611) (0.6760,0.1535) (0.6770,0.1552) (0.6780,0.1525) (0.6790,0.1618) (0.6800,0.1470) (0.6810,0.1423) (0.6820,0.1480) (0.6830,0.1546) (0.6840,0.1495) (0.6850,0.1465) (0.6860,0.1441) (0.6870,0.1403) (0.6880,0.1434) (0.6890,0.1376) (0.6900,0.1412) (0.6910,0.1354) (0.6920,0.1425) (0.6930,0.1351) (0.6940,0.1350) (0.6950,0.1292) (0.6960,0.1244) (0.6970,0.1307) (0.6980,0.1266) (0.6990,0.1206) (0.7000,0.1297) (0.7010,0.1309) (0.7020,0.1290) (0.7030,0.1255) (0.7040,0.1208) (0.7050,0.1249) (0.7060,0.1147) (0.7070,0.1211) (0.7080,0.1179) (0.7090,0.1253) (0.7100,0.1134) (0.7110,0.1229) (0.7120,0.1146) (0.7130,0.1139) (0.7140,0.1137) (0.7150,0.1145) (0.7160,0.1114) (0.7170,0.1160) (0.7180,0.1095) (0.7190,0.1122) (0.7200,0.1091) (0.7210,0.1093) (0.7220,0.1055) (0.7230,0.1080) (0.7240,0.1057) (0.7250,0.1091) (0.7260,0.1023) (0.7270,0.1049) (0.7280,0.1006) (0.7290,0.1075) (0.7300,0.0944) (0.7310,0.0954) (0.7320,0.0996) (0.7330,0.0935) (0.7340,0.0972) (0.7350,0.1008) (0.7360,0.0936) (0.7370,0.0973) (0.7380,0.0904) (0.7390,0.0926) (0.7400,0.0899) (0.7410,0.0902) (0.7420,0.0876) (0.7430,0.0870) (0.7440,0.0878) (0.7450,0.0880) (0.7460,0.0829) (0.7470,0.0853) (0.7480,0.0831) (0.7490,0.0844) (0.7500,0.0815) (0.7510,0.0837) (0.7520,0.0795) (0.7530,0.0808) (0.7540,0.0790) (0.7550,0.0795) (0.7560,0.0756) (0.7570,0.0804) (0.7580,0.0860) (0.7590,0.0776) (0.7600,0.0781) (0.7610,0.0726) (0.7620,0.0723) (0.7630,0.0747) (0.7640,0.0746) (0.7650,0.0714) (0.7660,0.0744) (0.7670,0.0707) (0.7680,0.0678) (0.7690,0.0683) (0.7700,0.0740) (0.7710,0.0680) (0.7720,0.0632) (0.7730,0.0663) (0.7740,0.0658) (0.7750,0.0675) (0.7760,0.0673) (0.7770,0.0624) (0.7780,0.0617) (0.7790,0.0647) (0.7800,0.0610) (0.7810,0.0608) (0.7820,0.0573) (0.7830,0.0604) (0.7840,0.0613) (0.7850,0.0638) (0.7860,0.0568) (0.7870,0.0598) (0.7880,0.0556) (0.7890,0.0550) (0.7900,0.0551) (0.7910,0.0553) (0.7920,0.0579) (0.7930,0.0515) (0.7940,0.0557) (0.7950,0.0514) (0.7960,0.0504) (0.7970,0.0518) (0.7980,0.0508) (0.7990,0.0503) (0.8000,0.0515) (0.8010,0.0510) (0.8020,0.0474) (0.8030,0.0518) (0.8040,0.0445) (0.8050,0.0470) (0.8060,0.0449) (0.8070,0.0442) (0.8080,0.0474) (0.8090,0.0458) (0.8100,0.0436) (0.8110,0.0480) (0.8120,0.0423) (0.8130,0.0413) (0.8140,0.0440) (0.8150,0.0416) (0.8160,0.0419) (0.8170,0.0405) (0.8180,0.0400) (0.8190,0.0371) (0.8200,0.0385) (0.8210,0.0374) (0.8220,0.0408) (0.8230,0.0386) (0.8240,0.0344) (0.8250,0.0365) (0.8260,0.0339) (0.8270,0.0347) (0.8280,0.0335) (0.8290,0.0339) (0.8300,0.0353) (0.8310,0.0346) (0.8320,0.0340) (0.8330,0.0353) (0.8340,0.0338) (0.8350,0.0357) (0.8360,0.0319) (0.8370,0.0329) (0.8380,0.0297) (0.8390,0.0307) (0.8400,0.0287) (0.8410,0.0294) (0.8420,0.0311) (0.8430,0.0307) (0.8440,0.0305) (0.8450,0.0293) (0.8460,0.0295) (0.8470,0.0293) (0.8480,0.0276) (0.8490,0.0278) (0.8500,0.0276) (0.8510,0.0252) (0.8520,0.0256) (0.8530,0.0255) (0.8540,0.0260) (0.8550,0.0255) (0.8560,0.0227) (0.8570,0.0263) (0.8580,0.0258) (0.8590,0.0221) (0.8600,0.0221) (0.8610,0.0221) (0.8620,0.0205) (0.8630,0.0219) (0.8640,0.0207) (0.8650,0.0216) (0.8660,0.0216) (0.8670,0.0220) (0.8680,0.0213) (0.8690,0.0194) (0.8700,0.0176) (0.8710,0.0221) (0.8720,0.0183) (0.8730,0.0193) (0.8740,0.0187) (0.8750,0.0164) (0.8760,0.0174) (0.8770,0.0164) (0.8780,0.0197) (0.8790,0.0161) (0.8800,0.0172) (0.8810,0.0167) (0.8820,0.0149) (0.8830,0.0172) (0.8840,0.0156) (0.8850,0.0146) (0.8860,0.0173) (0.8870,0.0156) (0.8880,0.0141) (0.8890,0.0127) (0.8900,0.0140) (0.8910,0.0129) (0.8920,0.0120) (0.8930,0.0137) (0.8940,0.0130) (0.8950,0.0127) (0.8960,0.0134) (0.8970,0.0113) (0.8980,0.0097) (0.8990,0.0115) (0.9000,0.0128) (0.9010,0.0085) (0.9020,0.0109) (0.9030,0.0105) (0.9040,0.0109) (0.9050,0.0102) (0.9060,0.0098) (0.9070,0.0102) (0.9080,0.0077) (0.9090,0.0083) (0.9100,0.0083) (0.9110,0.0086) (0.9120,0.0087) (0.9130,0.0075) (0.9140,0.0084) (0.9150,0.0069) (0.9160,0.0077) (0.9170,0.0053) (0.9180,0.0065) (0.9190,0.0078) (0.9200,0.0078) (0.9210,0.0067) (0.9220,0.0079) (0.9230,0.0071) (0.9240,0.0057) (0.9250,0.0059) (0.9260,0.0066) (0.9270,0.0053) (0.9280,0.0063) (0.9290,0.0045) (0.9300,0.0052) (0.9310,0.0042) (0.9320,0.0048) (0.9330,0.0047) (0.9340,0.0047) (0.9350,0.0049) (0.9360,0.0049) (0.9370,0.0037) (0.9380,0.0031) (0.9390,0.0033) (0.9400,0.0039) (0.9410,0.0018) (0.9420,0.0027) (0.9430,0.0032) (0.9440,0.0037) (0.9450,0.0032) (0.9460,0.0024) (0.9470,0.0027) (0.9480,0.0035) (0.9490,0.0021) (0.9500,0.0027) (0.9510,0.0025) (0.9520,0.0028) (0.9530,0.0020) (0.9540,0.0018) (0.9550,0.0019) (0.9560,0.0019) (0.9570,0.0014) (0.9580,0.0021) (0.9590,0.0018) (0.9600,0.0016) (0.9610,0.0018) (0.9620,0.0018) (0.9630,0.0009) (0.9640,0.0018) (0.9650,0.0014) (0.9660,0.0012) (0.9670,0.0014) (0.9680,0.0014) (0.9690,0.0018) (0.9700,0.0016) (0.9710,0.0011) (0.9720,0.0009) (0.9730,0.0008) (0.9740,0.0004) (0.9750,0.0006) (0.9760,0.0005) (0.9770,0.0007) (0.9780,0.0006) (0.9790,0.0005) (0.9800,0.0005) (0.9810,0.0006) (0.9820,0.0002) (0.9830,0.0002) (0.9840,0.0005) (0.9850,0.0000) (0.9860,0.0001) (0.9870,0.0004) (0.9880,0.0002) (0.9890,0.0001) (0.9900,0.0001) (0.9910,0.0001) (0.9920,0.0000) (0.9930,0.0000) (0.9940,0.0000) (0.9950,0.0000) (0.9960,0.0000) (0.9970,0.0000) (0.9980,0.0000) (0.9990,0.0000)};
	\end{axis}
	\end{tikzpicture}
	\caption{Experimental probability density function of $Q$}\label{fig:SharkFin}
\end{figure}

To test the simulation itself we ran an experiment for the mean area.  The observed average value of the ratio $\myQ = |\Delta RST|/|\Delta ABC|$ is expected to approach $\mu_1=1/4$ as the sample size is increased. A Java application was employed to study the deviation of the experimental average from its theoretical value, $\mathrm{err} = E_{exp}[\myQ] - E[\myQ]$. From Central Limit Theorem $\mathrm{err}$ has an approximately normal distribution with a standard deviation of $\sigma/\sqrt{n}$, where $\sigma =\sqrt{\mu_2-\mu_1^{2}}=1/4\sqrt 3$, and $n$ is the sample size. As such, $E_{exp}[|\mathrm{err}|]$, the experimental average value of $|\mathrm{err}|$,    is to approach $\sqrt{2/n\pi} \sigma$. 
We ran the simulation with sample of sizes of $n=10^2$ to $10^8$ and averaged $|\mathrm{err}|$ over $50$ trials. The result is  displayed in Table \ref{table:MonteCarloTest}.

\begin{comment}
\begin{table}[h]
	\centering
	\begin{tabular}{ | c | c | c |}
		\hlx{hv}
		$n$ & $E_{exp}[|\mathrm{err}|]$  & $\sqrt{2/n\pi} \sigma$   \\
		% $n$=sample size & err &  $\sigma/\sqrt{n}$\\
		\hlx{vhvv}
		$10^2$ & $-1.4\times10^{-2}$  & $1.4\times 10^{-2}$\\
		$10^3$ & $\phantom{-}6.7\times10^{-3}$ & $4.5\times 10^{-3}$ \\
		$10^4$ & $\phantom{-}5.3\times10^{-4}$ & $1.4\times 10^{-3}$\\
		$10^5$ & $-1.5\times10^{-4}$ & $4.5\times 10^{-4}$\\
		$10^6$ & $\phantom{-}7.1\times10^{-5}$&$1.4\times 10^{-4}$ \\
		$10^7$ & $-6.1\times10^{-5}$& $4.5\times 10^{-5}$ \\
		$10^8$ & $-1.4\times10^{-5}$ &$1.4\times 10^{-5}$\\
		\hline
	\end{tabular}
	\caption{Monte Carlo Data}\label{table:MonteCarlo}
\end{table}
\end{comment}

\begin{table}[H]
	\centering
	\begin{tabular}{ c | c  c  c  c  c  c  c  }
		$n$ & $10^2$ & $10^3$ & $10^4$ & $10^5$ & $10^6$ & $10^7$ & $10^8$ \\\\
$E_{exp}[|err|]$ & 1.3E-2 & 3.0E-3 & 1.1E-3& 3.8E-4&  1.2E-4& 3.8E-5&1.2E-5 \\
$\sqrt{2/n\pi}\sigma$ &1.1E-2&3.6E-3&1.1E-3&3.6E-4&1.1E-4&3.6E-5&1.1E-5\\
\end{tabular}
	\caption{A test of the Monte Carlo simulation.\label{table:MonteCarloTest}}
\end{table}

Note that the average observed error decreases approximately by a factor of 10 for every increase of a factor of 100 in sample size.

\begin{comment}
10^2: 1.3236726778772518E-2
10^3:  3.009873250385698E-3
10^4 : 1.1399255511913924E-3
10^5 : 3.809102729615965E-4
10^6 : 9.636680416315157E-5  1.1621989398601817E-4
10^7 : 3.85157931698632E-5
10^8 : 1.1853210084296628E-5
\end{comment}

\section{Cumulative and Probability Density Functions}
In this section we will derive the cumulative density function, $\CDF$, and the probability density function, $\PDF$, of the area quotient $\myQ$.

We have $\CDF(c) = \mathrm{Vol} \left\{ (r,s,t) \in [0,1]^3 \mid \myQ \leq c \right\}$. By a rotation of coordinate system one sees that $Q^{-1}(c)$ is a hyperboloid. For $c \in [0,1/4)$ the surface is a hyperboloid of one sheet, for $c=1/4$ it is a double cone, and for $c \in (1/4,1]$ it is a hyperboloid of two sheets. For $c<1/4$, $\CDF(c)$ is equal to the volume of a region similar to Figure \ref{fig:VolumesA}. For $c>1/4$ it is equal to the volume of a region similar to Figure \ref{fig:VolumesB}.

\begin{figure}[h]
	\begin{subfigure}{0.5\textwidth}
		\includegraphics[scale=0.35]{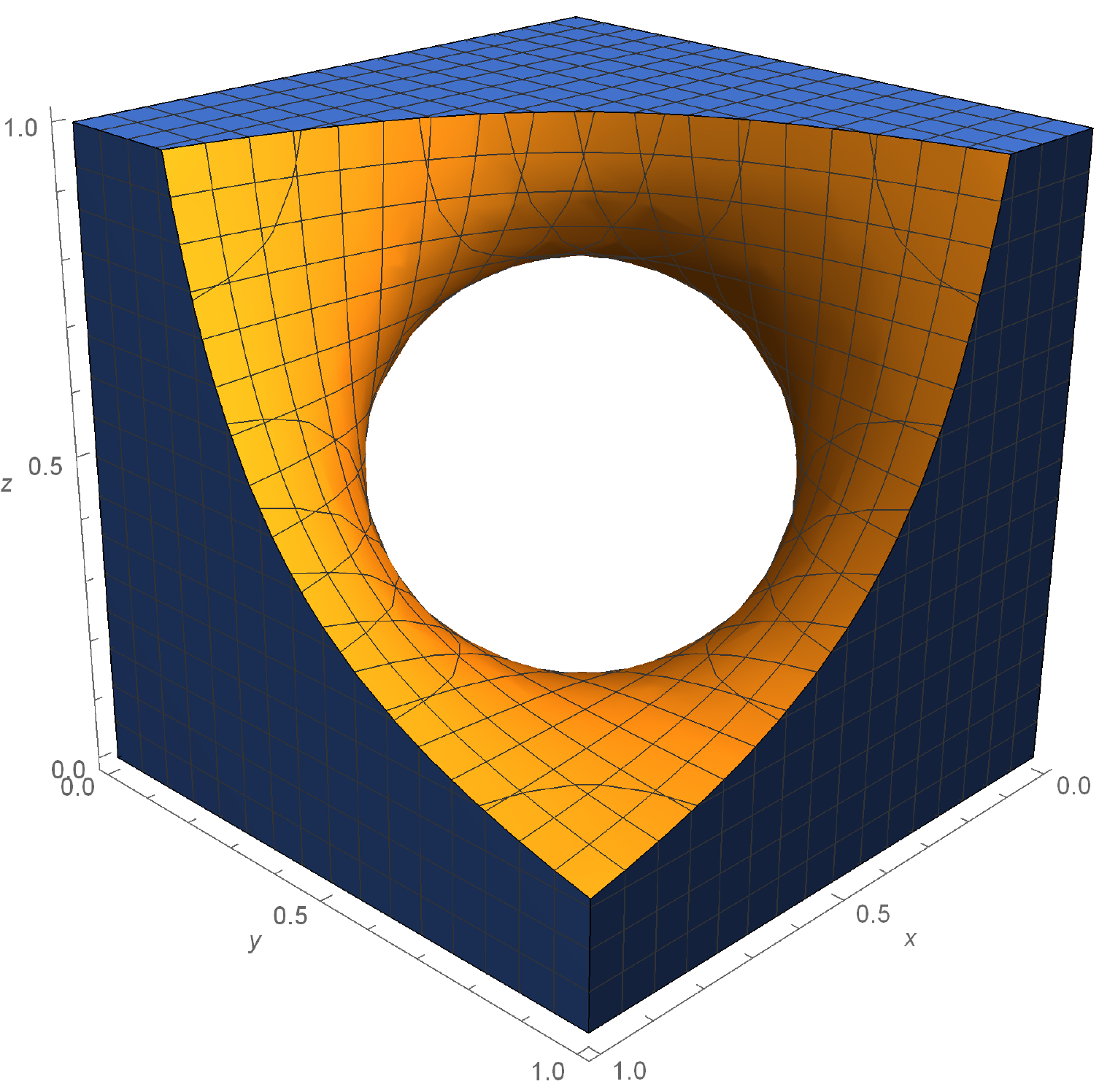}
		\caption{$c \in [0,1/4)$} \label{fig:VolumesA}
	\end{subfigure}
	\begin{subfigure}{0.5\textwidth}
		\includegraphics[scale=0.35]{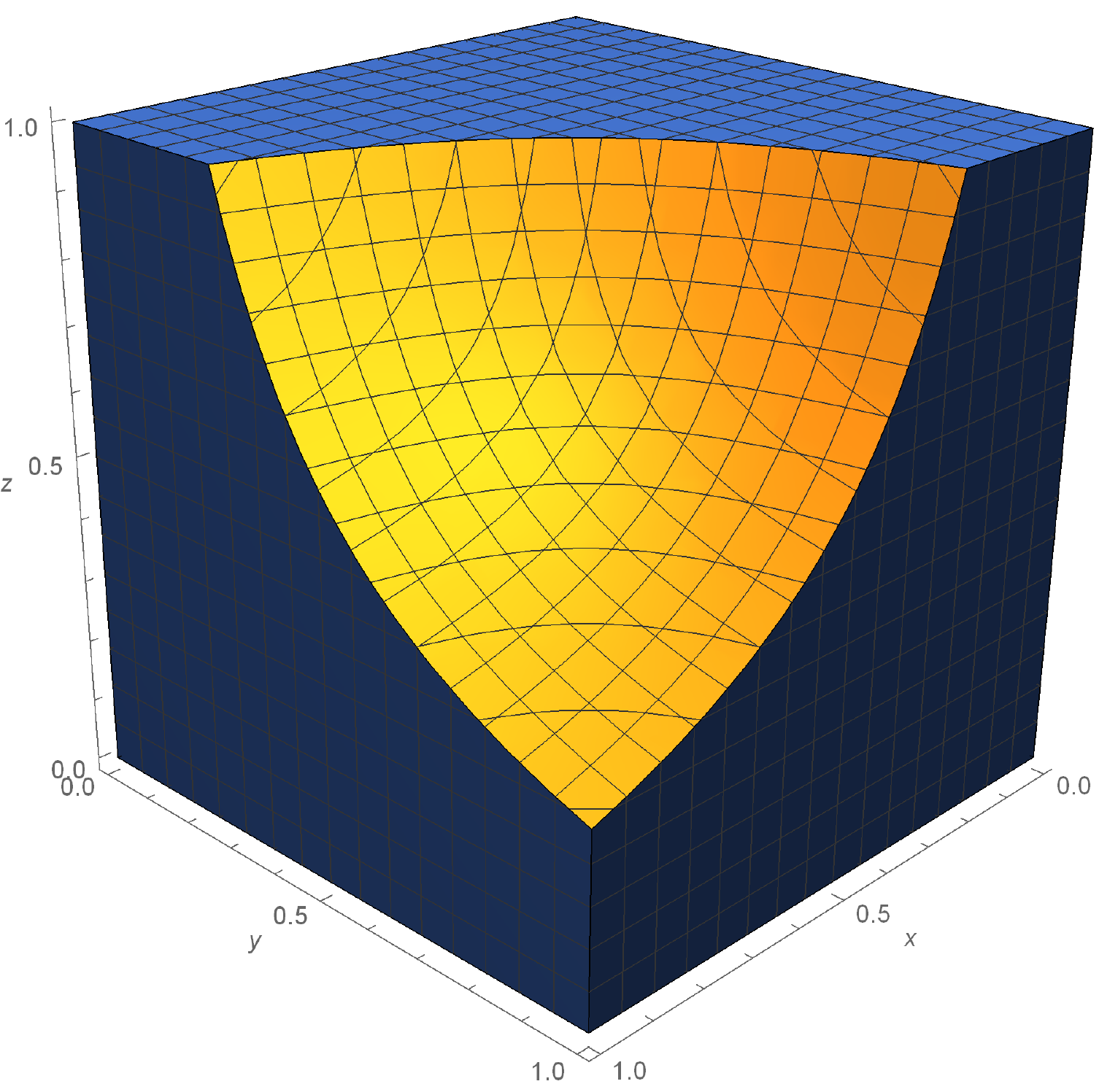}
		\caption{$c \in (1/4,1]$} \label{fig:VolumesB}
	\end{subfigure}
	\caption{Contour plots of $\myQ$ for $c=1/5$ and $1/3$ respectively.}\label{fig:Volumes}
\end{figure}

To visualize the integration volume we may utilize the following Mathematica command:
\begin{center}
	\begin{lstlisting} [caption={Volume of displayed region is $\CDF(c)$}]
	Q[x_,y_,z_]:=(1-x)*(1-y)*(1-z)+x*y*z
	Manipulate[ RegionPlot3D[ Q[x,y,z]<c,{x,0,1},{y,0,1},{z,0,1},
	AxesLabel->{X,Y,Z} ],{c,0,1,0.01} ]
	\end{lstlisting}
\end{center}
and to see the slices used in the integration process we may utilize the following:
\begin{center}
	\begin{lstlisting} [caption={Intersection of slicing planes  with the hyperboloid}]
	Q[x_,y_,z_]:=(1-x)*(1-y)*(1-z)+x*y*z
	Manipulate[ ContourPlot3D[  {Q[x,y,z]==c,x==a},  {x,0,1},{y,0,1},{z,0,1},
	AxesLabel->{X,Y,Z} ], {c,0,1,0.01}, {a,0,1,0.001} ]
	\end{lstlisting}
\end{center}

The derivation of $\CDF(c)$ involves many integration steps, mostly of the type $\int P(t) \ln(R(t)) dt$, where $P$ and $R$ are polynomials. These can be done by integration by parts. We employed Mathematica’s integration routine followed by hand simpliﬁcation. The summary is displayed here, and the derivation is detailed in the next section. For the cumulative density function we find 
\begin{equation}
	\CDF(c) = \left\{\begin{array}{ll}
	c-(3c-\frac{1}{2})\ln c+(1-4c)^{3/2}\tanh^{-1}\sqrt{1-4c}, & \text{for } 0\leq c \leq \frac{1}{4}\\
	\frac{1}{4}(1+\ln4 ), & \text{for } c=\frac{1}{4}\\
	c-(3c-\frac{1}{2})\ln c+(4c-1)^{3/2}(\tan^{-1}\sqrt{4c-1}-\frac{\pi}{3}) & \text{for } \frac{1}{4}\leq c\leq 1.
	\end{array}\right.
\end{equation}
By differentiating $\CDF(c)$ we arrive at
\begin{equation}\label {eq: PDFformula}
	\PDF(c) = \left\{\begin{array}{ll}
	-3\ln c -6\sqrt{1-4c}\tanh^{-1}\sqrt{1-4c}, & \text{for } 0\leq c \leq \frac{1}{4}\\
	3\ln 4, & \text{for } c=\frac{1}{4}\\
	-3\ln c + 2\sqrt{4c-1}\left(-\pi+3\tan^{-1}\sqrt{4c-1}\right) & \text{for } \frac{1}{4}\leq c\leq 1.
	\end{array}\right.
\end{equation}
The graphs of these two distributions are displayed in Figure \ref{fig:DisPlots}.
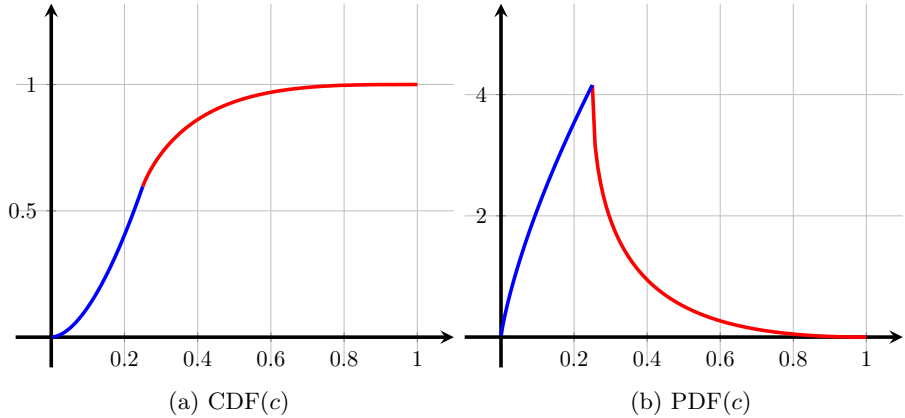
\begin{figure}[h]
	\begin{subfigure}{0.5\textwidth}
		\begin{tikzpicture}[scale = 0.85]
		\begin{axis}[ultra thick, grid=both,
		xmax=1,ymax=1.2,
		axis lines=middle,
		restrict y to domain=-7:12,
		enlargelimits]
		\addplot[blue,domain=0:1/4,samples=100]  {\x-(3*\x-(1/2))*ln(\x)+pow(1-4*\x,3/2)*arctanh(sqrt(1-4*\x))}; 
		\addplot[red,domain=1/4:1,samples=100]  {\x-(3*\x-(1/2))*ln(\x)+pow(4*\x-1,3/2)*((3.141592/180)*atan(sqrt(4*\x-1))-(3.141592/3))};
		\end{axis}
		\end{tikzpicture}
		\caption{$\CDF(c)$} \label{fig:DisPlotsB}
	\end{subfigure}
	\begin{subfigure}{0.5\textwidth}
		\begin{tikzpicture}[scale = 0.85]
		\begin{axis}[ultra thick, grid=both,
		xmax=1,ymax=5,
		axis lines=middle,
		restrict y to domain=-7:12,
		enlargelimits]
		\addplot[red,domain=1/4:1,samples=100]  {-3*ln(\x)+2*sqrt(4*\x-1)*(-3.1415926+3*(3.141592/180)*atan(sqrt(4*\x-1)))}; 
		\addplot[blue,domain=0:1/4,samples=500]  {-3*ln(\x)-6*sqrt(1-4*\x)*arctanh(sqrt(1-4*\x))};
		\end{axis}
		\end{tikzpicture}
		\caption{$\PDF(c)$} \label{fig:DisPlotsA}
	\end{subfigure}
	\caption{Plots of $\CDF(c)$ and $\PDF(c)$} \label{fig:DisPlots}
\end{figure}

We verify that the experimental PDF, Figure \ref{fig:SharkFin},  is close to exact result \eqref{eq: PDFformula} and Figure \ref{fig:DisPlotsA}.

\subsection{Derivation of \boldmath{$\CDF(c)$} and \boldmath{$\PDF(c)$} for \boldmath{$c \in (1/4,1]$}}
Given a fixed value of $c \in (1/4,1]$, the inscribed triangle's area, we can rewrite $\myQ = (rst) + (1-r)(1-s)(1-t) = c$ as $r(s,t,c)$, a function of $s$, $t$, and $c$ as follows
\begin{equation}
	r(s,t,c) = \frac{c-1-st+s+t}{s+t-1}.
\end{equation}
Note that when $r=1$ we get $st=c$ and as a result the integration limits at a fixed $t$ will be from $s=c/t$ to $s=1$, and $t$ will have a range from $c$ to 1. 
The volume of the region in Figure \ref{fig:VolumesA} can be calculated through its complement
\begin{comment}
\begin{figure}[h]
\centering
\includegraphics[scale=0.35]{CDF_Left.pdf}
\includegraphics[scale=0.35]{CDF_Right.pdf}
\caption{Figure 4a and Figure 4b, for $c \in [0,1/4)$, and $c \in (1/4,1]$ respectively}
\end{figure}
\end{comment} 
\begin{equation}
	\CDF(c) = 1-2\int_{c}^{1}\int_{\tfrac{c}{t}}^{1}(1-r(s,t,c)) ds dt.
\end{equation}
To perform this calculation we use the Mathematica command
\begin{center}
	\begin{lstlisting}[caption={Mathematica integration code for $c \in (1/4,1)$  }]
	r[s_,t_,c_]:=(c-1-s*t+s+t)/(s+t-1)
	FullSimplify[Assuming[1>c>1/4,
	1-2*Integrate[1-r[s,t,c],{t,c,1},{s,c/t,1}] ]]
	\end{lstlisting}\label{lst: integrationR}
\end{center}
\begin{comment}
$$\textrm{ FullSimplify[Assuming[1$>$c$>$1/4,1-2*Integrate[1-r[s,t,c],\{t,c,1\},\{s,c/t,1\}]]]}$$
\end{comment}
and we arrive at the following expression 
\begin{equation}
	\CDF(c) = c-(3c-\frac{1}{2})\ln c -\frac{(4c-1)^{\tfrac{3}{2}}}{3}\left(\tan^{-1}\left(\frac{1}{\sqrt{4c-1}}\right)-\tan^{-1}\left(\frac{2c-1}{\sqrt{4c-1}}\right)\right),
\end{equation}
which can be simplified, using a trig identity explained below, to produce
\begin{equation}
	\CDF(c) = c-(3c-\frac{1}{2})\ln c + (4c-1)^{\tfrac{3}{2}}\left(\tan^{-1}(\sqrt{4c-1})-\frac{\pi}{3}\right).
\end{equation}
By differentiation we can find $\PDF(c)$
\begin{equation}
	\PDF(c) = \frac{d}{dc}\CDF(c) = 2\sqrt{4c-1}\left(3\tan^{-1}\sqrt{4c-1}-\pi\right)-3\ln c.
\end{equation}
\begin{lemma}[A Machin-like Identity]
	For $c>\frac{1}{4}$ we have
	\begin{equation}
		\tan^{-1}\left(\frac{1}{\sqrt{4c-1}}\right)-\tan^{-1}\left(\frac{2c-1}{\sqrt{4c-1}}\right) = \pi-3\tan^{-1}\sqrt{4c-1}.
	\end{equation}
\end{lemma}
Note that  the derivatives of both sides are equal to $-(3/2) c^{-1} (4c-1)^{-1/2}$, and  for $c=1$ both sides are equal to zero. Hence the identity is valid.
%Note that this identity is valid for $c=1$. To prove this identity, we can apply the tangent function to both sides. On the left we use the identity $\tan(x-y) = \frac{\tan(x)-\tan(y)}{1+\tan(x)\tan(y)}$. On the right side we apply $\tan(3x) = \frac{3\tan(x)-\tan^3(x)}{1-3\tan^2(x)}$. We confirm that both sides become $\frac{c-1}{1-3c}\sqrt{4c-1}$.
\subsection{Derivation of CDF(c) and PDF(c) for \boldmath{$c \in [0,1/4)$}}
Due to the presence of a hole in the middle of the corresponding  volume, this integration is more involved than the previous case.
To delegate the segmentation of the integral to Mathematica we may use the Boole command, then the calculation of $\CDF(c)$ for $c \in [0,1/4)$ can be done via the following
\begin{lstlisting}[caption={Mathematica Integration for $c \in (0,1/4)$} \label{lst: integrationR}]
Assuming[0<c<1/4,
Integrate[Boole[x*y*z+(1-x)*(1-y)*(1-z)<=c],{x,0,1},{y,0,1},{z,0,1}]]
FullSimplify[% // TrigToExp, 0 < c < 1/4]
\end{lstlisting}
Which results in
\begin{equation}
	\begin{split}
		&\CDF(c) = \frac{1}{12\sqrt{1-4c}}\Big[12c(\sqrt{1-4c}-12c\log(2)+\log(32))\\
		&+12c(4c-1)\log(1-\sqrt{1-4c})+4c(28c-11)\log(1+\sqrt{1-4c})\\
		&+3\log(1+\sqrt{1-4c}-2c)+4\log\left(\frac{1+\sqrt{1-4c}}{c}\right)+6\sqrt{1-4c}\log(c)\\
		&-2c((20c-7)\log(1-\sqrt{1-4c}-2c)+(9-12c)\log(1+\sqrt{1-4c}-2c)\\
		&+2(9\sqrt{1-4c}+16c-8)\log(c))+\log\left(\frac{-1}{64(\sqrt{1-4c}+2c-1)}\right)\Big].
	\end{split}
\end{equation}
To simplify the above, notice that $(1 \pm \sqrt{1-4c})^2 = 2(1-2c \pm \sqrt{1-4c})$, and $\tanh^{-1}(a) = \frac{1}{2}\ln(\frac{1+a}{1-a})$. After further algebraic simplification, the equation becomes
\begin{equation}
	\CDF(c) = c-(3c-\frac{1}{2})\ln c+(1-4c)^{3/2}\tanh^{-1}\sqrt{1-4c}.
\end{equation}
Upon differentiation of $\CDF(c)$ the PDF is found to be
\begin{equation}
	\PDF(c) = -3\ln c -6\sqrt{1-4c}\tanh^{-1}\sqrt{1-4c}.
\end{equation}

\begin{comment}
$\frac{12 c \left(\sqrt{1-4 c}-12 c \log (2)+\log (32)\right)+4 \log \left(\frac{\sqrt{1-4 c}+1}{c}\right)-2 c \left(2 \left(16 c+9 \sqrt{1-4 c}-8\right) \log (c)+(20 c-7) \log \left(-2 c-\sqrt{1-4 c}+1\right)+(9-12 c) \log \left(-2 c+\sqrt{1-4 c}+1\right)\right)+\log \left(-\frac{1}{64 \left(2 c+\sqrt{1-4 c}-1\right)}\right)+3 \log \left(-2 c+\sqrt{1-4 c}+1\right)+12 c (4 c-1) \log \left(1-\sqrt{1-4 c}\right)+4 c (28 c-11) \log \left(\sqrt{1-4 c}+1\right)+6 \sqrt{1-4 c} \log (c)}{12 \sqrt{1-4 c}}$
\end{comment}
%\textrm{Assuming[0<c<1/4,Integrate[\\Boole[x*y*z+(1-x)*(1-y)*(1-z)\leq c],\{x,0,1\},\{y,0,1\},\{z,0,1\}]]}

\section*{Future Research Directions}
We will  extend the current findings  in several directions. First, the case of tetrahedron inscribed-tetrahedron picking appears as a natural extension. Next, the number theoretic properties of the integer sequence in \eqref{nthmoment}  will be investigated. Finally we notice that when we extend $\PDF(c)$ from $(1/4,1]$, as a complex function, to $[0,1/4)$ then its real part is same as the $\PDF(c)$ for $[0,1/4)$. An explanation of this phenomena would be of interest.

\section*{Acknowledgments}
I thank Dr. Patterson for supervising this research.

\end{document}